\documentclass[11pt]{amsart}
\usepackage{amscd}
\usepackage{amsmath, latexsym, amssymb}
\usepackage[all]{xy}
\usepackage{graphicx}
\usepackage{amsmath}
\usepackage{amsfonts}
\usepackage{amssymb}
\usepackage{latexsym}
\usepackage{dsfont}
\usepackage{slashed}
\usepackage{mathdots}

\usepackage{soul}
\usepackage{comment}
\usepackage{color}
\usepackage{rotating}
\usepackage{url}

\usepackage{longtable}
\usepackage[table]{xcolor}

\usepackage{tikz}
\usetikzlibrary{arrows,decorations.markings}
\usetikzlibrary{matrix}
\usetikzlibrary{graphs}
\usetikzlibrary{backgrounds}

\newcommand{\ttt}{{\mathfrak t}}
\newcommand{\uuu}{{\mathfrak u}}
\renewcommand{\sp}{\mathfrak{\mathop{sp}}}

\definecolor{grey}{RGB}{128,128,128}

\numberwithin{equation}{section}
\theoremstyle{plain}
\newtheorem{lemma}[equation]{Lemma}
\newtheorem{proposition}[equation]{Proposition}
\newtheorem{theorem}[equation]{Theorem}
\newtheorem{corollary}[equation]{Corollary}

\theoremstyle{definition}
\newtheorem{definition}[equation]{Definition}
\newtheorem{remark}[equation]{Remark}
\newtheorem{example}[equation]{Example}

\newtheorem*{proposition*}{Proposition}
\newtheorem*{lemma*}{Lemma}

\newcommand{\R}{{\mathds R}}
\newcommand{\C}{{\mathds C}}

\newcommand{\Z}{{\mathds Z}}

\newcommand{\Aa}{{\mathcal A}}

\newcommand{\Ff}{{\mathcal F}}
\newcommand{\Gg}{{\mathcal G}}    

\newcommand{\Nn}{{\mathcal N}}
\newcommand{\Oo}{{\mathcal O}}

\newcommand{\Zz}{{\mathcal Z}}

\newcommand{\ad}{{\rm ad \,}}

\renewcommand{\a}{{\mathfrak a}}

\newcommand{\g}{{\mathfrak g}}

\newcommand{\h}{{\mathfrak h}}

\renewcommand{\so}{{\mathfrak{so}}}

\newcommand{\ssl}{{\mathfrak {sl}}}

\renewcommand{\t}{{\mathfrak t }}

\newcommand{\GL}{{\rm GL}}
\newcommand{\SL}{{\rm SL}}

\newcommand{\hs}{\kern 0.8pt}
\newcommand{\hsss}{\kern 1.2pt}
\newcommand{\hm}{\kern -0.8pt}
\newcommand{\hmm}{\kern -1.2pt}

\newcommand{\hssh}{\kern 1.2pt}
\newcommand{\hshs}{\kern 1.6pt}
\newcommand{\hssss}{\kern 2.0pt}

\newcommand{\upgam}{\hs^\gamma\kern-0.5pt }
\newcommand{\upalpha}{\hs^\alpha\kern-0.5pt}

\newcommand{\wG}{\widehat{G}}
\newcommand{\wg}{\hat{\g}}

\newcommand{\Gal}{{\rm Gal}}
\newcommand{\Zm}{{\mathcal Z}}

\newcommand{\SmallMatrix}[1]{\text{{\tiny\arraycolsep=0.4\arraycolsep\ensuremath
			{\begin{pmatrix}#1\end{pmatrix}}}}}

\newcommand{\Ho}{{\mathrm{H}\kern 0.3pt}}
\newcommand{\Zl}{{\mathrm{Z}\kern 0.2pt}}
\newcommand{\Bd}{{\mathrm{B}\kern 0.2pt}}

\newcommand{\diag}{{\rm diag}}
\newcommand{\reg}{\mathrm{reg}}

\newcommand{\cc}{\raise 1.7pt \hbox{\Tiny{$\bullet$}}}

\newcommand{\Nm}{{\mathcal N}}

\newcommand{\cC}{{\scriptscriptstyle{\C}}}

\def\upgam{{\hs^\gamma\hm}}

\def\Fm{{\mathcal F}}

\def\Zgg{{\mathfrak z}}

\begin{document}
	
	
	\def\DynkinNodeSize{1.5mm}
	\def\DynkinArrowLength{1.5mm}
	\tikzset{
		dnode/.style={
			circle,
			inner sep=0pt,
			minimum size=\DynkinNodeSize,
			fill=white,
			draw},
		middlearrow/.style={
			decoration={markings,
				mark=at position 0.6 with
				{\draw (0:0mm) -- +(+135:\DynkinArrowLength); \draw (0:0mm) -- +(-135:\DynkinArrowLength);},
			},
			postaction={decorate}
		},
		leftrightarrow/.style={
			decoration={markings,
				mark=at position 0.999 with
				{
					\draw (0:0mm) -- +(+135:\DynkinArrowLength); \draw (0:0mm) -- +(-135:\DynkinArrowLength);
				},
				mark=at position 0.001 with
				{
					\draw (0:0mm) -- +(+45:\DynkinArrowLength); \draw (0:0mm) -- +(-45:\DynkinArrowLength);
				},
			},
			postaction={decorate}
		},
		sedge/.style={
		},
		dedge/.style={
			middlearrow,
			double distance=0.5mm,
		},
		tedge/.style={
			middlearrow,
			double distance=1.0mm+\pgflinewidth,
			postaction={draw}, 
		},
		infedge/.style={
			leftrightarrow,
			double distance=0.5mm,
		},
	}

\title[Real semisimple $\Z_m$-graded  Lie algebras]%
{Semisimple elements and the little Weyl group of real semisimple $\Z_m$-graded  Lie algebras}

\author{Willem de Graaf}
\address{Department of Mathematics, University of Trento, Povo (Trento), Italy}
\email{degraaf@science.unitn.it}
	
\author{H\^ong V\^an L\^e}
\address{Institute of Mathematics, Czech  Academy of Sciences,
	Zitna 25, 11567 Praha 1, Czech Republic}
\email{hvle@math.cas.cz}

\thanks{Research  of L\^e was supported  by Institute of Mathematics, Czech Academy of Sciences (RVO: 67985840) and GA\v CR-project  22-00091S}
	
	\date{\today}

\begin{abstract}
We consider the semisimple orbits of a Vinberg $\theta$-representation.
First we take the complex numbers as base field.
By a case by case analysis we show a technical result stating the equality
of two sets of
hyperplanes, one corresponding to the restricted roots of a Cartan subspace,
the other corresponding to the complex reflections in the (little) Weyl group.
The semisimple orbits have representatives in a
finite number of sets that correspond to reflection subgroups of the 
(little) Weyl group. One of the consequences of our technical result is that
the elements in a fixed such set all have the same
stabilizer in the acting group.
Secondly we study what happens when the base field is the real numbers. We
look at Cartan subspaces and show that the real Cartan subspaces can be
classified by the first Galois cohomology set of the normalizer of a fixed
real Cartan subspace. In the real case 
the orbits can be classified using Galois cohomology. However, in order for
that to work we need to know which orbits have a real representative. We show
a theorem that characterizes the orbits of homogeneous semisimple  elements that do have such a real representative.
This closely follows and generalizes a theorem from \cite{bgl}. 
\end{abstract}



\maketitle

\section{Introduction}

In general it is very difficult, if not impossible, to classify in a meaningful
way the orbits of a complex algebraic group acting on a vector space. A notable
exception to this rule is formed by the so-called $\theta$-representations
introduced and studied by Vinberg in the 70's, \cite{Vinberg1976,Vinberg1979}.
These are representations that are constructed from a cyclic grading of a
semisimple Lie algebra. They share many properties with the adjoint
representation of a semisimple algebraic group on its Lie algebra. In
particular, the elements of the module have a Jordan decomposition. This divides
the elements, and hence the orbits, into three groups: nilpotent, semisimple and
mixed. For classifying the nilpotent orbits Vinberg developed a method based
on classifying regular subalgebras of a certain type, called carrier algebras
\cite{Vinberg1979}. There are also algorithms for this purpose based on
embedding a nilpotent element in an $\ssl_2$-triple, \cite{gra11}.

For studying the semisimple orbits the key concept is that of a Cartan subspace
\cite{Vinberg1976}. This is a maximal subspace consisting of commuting
semisimple elements. It plays a role that in many ways is similar to that of
a Cartan subalgebra of a semisimple Lie algebra. In particular, every semisimple
orbit (that is, orbit consisting of semisimple elements) has at least one
point in a given Cartan subspace. Moreover, two elements of a fixed Cartan
subspace lie in the same orbit if and only if they are conjugate relative
to a finite group, called the Weyl group (or little Weyl group)
of the graded semisimple Lie algebra. So one could consider the classification
problem of the semisimple elements to be solved once a Cartan subspace and
corresponding Weyl group have been determined. However, in a particular
example corresponding to a grading of the simple Lie algebra of type $E_8$,
Vinberg and Elashvili have shown that this classification can be refined
\cite{VE1978}. They define seven subsets of a fixed Cartan subspace which
correspond to seven conjugacy classes of reflection subgroups of the Weyl group
of the grading. They show the following:
\begin{enumerate}
\item Every semisimple orbit has a point in exactly one of the seven subsets.
\item Two elements of the same subset lie in the same orbit if and only if
  they are conjugate under an explicitly given finite group.
\item All elements of the same subset have the {\em same} stabilizer in the
  acting algebraic group.
\end{enumerate}
Furthermore, this more refined classification makes it possible to
also classify the mixed elements. We also remark that subsequently a number of
orbit classifications of $\theta$-representations have appeared in the
literature, cf. \cite{an81}, \cite{AE1982}, \cite{gatim}.

In Section \ref{sec:fams} we consider the semisimple orbits in a
$\theta$-representation over the complex field. Let $W$ denote the Weyl group
corresponding to a fixed Cartan subspace $\h^\cC$. Then $W$ is generated by
complex reflections and each complex reflection in $W$ leaves a hyperplane
in $\h^\cC$ pointwise fixed. Denote the set of hyperplanes so obtained by
$A$. The Cartan subspace $\h^\cC$ is contained in a Cartan subalgebra of the
ambient Lie algebra. This Cartan subspace gives rise to a root system.
The kernel of each root when restricted to $\h^\cC$ (when it is not the whole
of $\h^\cC$) is a hyperplane in $\h^\cC$. Denote the set of such hyperplanes
by $B$. We first show by a case by case analysis that $A=B$.
It is always possible to define a finite number of subsets, each corresponding
to a reflection subgroup of the Weyl group, of a given Cartan
subspace having the properties (1), (2) above. From $A=B$ we deduce that 
these subsets also have property (3).

We also look at real forms of a $\theta$-representation. We do not attempt to
define this concept, but illustrate it with an example.
The $\theta$-representation studied in \cite{VE1978} is equivalent to
the action of $\SL(9,\C)$ on the space $\wedge^3 \C^9$. An important real
form of this is the action of $\SL(9,\R)$ on $\wedge^3 \R^9$. In
\cite{bgl} the orbits of this real form are classified, using the
fact that the complex version is constructed as a $\theta$-representation
and using Galois cohomology. For the nilpotent orbits this is rather
straightforward. In many cases (for example those studied in
\cite{VE1978,gatim}) all complex nilpotent orbits have real representatives.
If such representatives are not known, then with the 
methods of \cite{bo21} it can be decided whether the orbit has a real
representative and find one in the affirmative case. Then the first Galois
cohomology set of of the stabilizer of a real representative $e$ yields the
real nilpotent orbits contained in the complex nilpotent orbit of $e$.

If there are semisimple orbits in $\g_1$ then there are infinitely many.
For this reason, classifying the real semisimple orbits offers extra
difficulties because we need a description of the orbits that have real
representatives.
It can happen that the orbit of a given element $v$ in a given
complex Cartan subspace has real points, without $v$ itself being real.
In \cite{bgl} a theorem is shown that explicitly characterizes this
situation. In Section \ref{sec:realss} we show that an analogous theorem holds
in general. This uses the classification of the semisimple orbits involving
a finite number of subsets with (1), (2), (3). Also we study the problem
of classifying real Cartan subspaces: we show that the Cartan subspaces,
up to conjugacy, correspond to the elements of the first Galois cohomology set
of the normalizer of a complex Cartan subspace that is defined over $\R$.

\section{Preliminaries}

\subsection{Vinberg's construction of certain representations of algebraic groups}

Here we briefly describe the construction of Vinberg's $\theta$-representations.
Throughout we write $\Z_m= \Z/m\Z$. 

Let $\g^\cC$ be a complex semisimple Lie algebra equipped with a $\Z_m$-grading
$$\g^\cC = \bigoplus_{i\in \Z_m} \g^\cC_i,$$
where the $\g^\cC_i$ are subspaces such that $[\g^\cC_i,\g^\cC_j]\subset \g^\cC_{i+j}$.

Let $G\subset \GL(\g^{\cC})$ be the identity component of the automorphism group of
$\g^\cC$.
The group $G$ is also known as the adjoint group of $\g^\cC$. It is generated,
as a group, by the maps $\exp(\ad x)$ where $x\in \g^\cC$ is nilpotent.
Its Lie algebra
is $\ad \g^\cC$ which is the Lie algebra of all adjoint maps $\ad x:
\g^\cC \to \g^\cC$, $\ad x(y)=[x,y]$ for $x,y\in \g^\cC$. In the sequel we
identify $\g^\cC$ and $\ad \g^\cC$ and we say that $\g^\cC$ is the Lie algebra
of $G$.

The subalgebra $\g^\cC_0\subset \g^\cC$ is reductive (see, e.g., \cite[Lemma 8.1(c)]{kac})
and hence there is a connected reductive subgroup $G_0$ of $G$ whose Lie
algebra is $\g^\cC_0$ (more precisely, it is $\ad \g^\cC_0$). 
The subspace $\g^\cC_1$ is naturally a $\g^\cC_0$-module,
and hence also a $G_0$-module.

The grading of $\g^\cC$ yields an automorphism $\theta$ of $\g^\cC$ of order $m$ by
setting $\theta(x) = \omega^i x$ for $x\in \g^\cC_i$, where $\omega\in\C$ is a
primitive $m$-th root of unity. Conversely, if $\theta$ is an automorphism of
$\g^\cC$ of order $m$ then letting $\g^\cC_i$ be the eigenspace of $\theta$ with
eigenvalue $\omega^i$ gives a $\Z_m$-grading of $\g^\cC$. Therefore, in the
sequel we denote a $\Z_m$-graded Lie algebra by $(\g^\cC,\theta)$. 

Let $\wG$ be a reductive complex algebraic group with Lie algebra $\wg^\cC$.
Suppose that there exists a surjective morphism $\psi : \wG \to G_0$, such
that its differential (denoted by the same letter) $\psi : \wg^\cC \to \g^\cC_0$
is an isomorphism. Then $\psi$ makes $\g_1^\cC$ into a $\wG$-module by
$g\cdot x = \psi(g)\cdot x$ for $g\in \wG$
and $x\in \g_1^\cC$. Let $V^\cC$ be a $\wG$-module such that there is an
isomorphism of
$\wG$-modules $\g_1^\cC\to V^\cC$. 
Then classifying the orbits of $\wG$ in $V^\cC$
is equivalent to classifying the orbits of $G_0$ in $\g_1^\cC$.

A first remark is that $\g_1^\cC$ is closed under Jordan decomposition.
This means the following. Let $x\in \g_1^\cC$. Then $x\in \g^\cC$ and hence there
are $s,n\in \g^\cC$ such that $s$ is semisimple, $n$ is nilpotent and 
$[s,n]=0$. Furthermore, $\omega s +\omega n=\omega x = \theta(x) = \theta(s)
+\theta(n)$ so that $s,n\in \g_1^\cC$. This implies that the orbits of $G_0$ in
$\g_1^\cC$ are divided into three groups: nilpotent, semisimple and mixed
(the latter are neither semisimple nor nilpotent).

\begin{example}\label{exa:trivec}
Let $\g^\cC$ be the simple complex Lie algebra of type $E_8$. There is an
automorphism $\theta$ of $\g^\cC$ of order 3 such that $\g_0^\cC$ is simple
of type $A_8$. This automorphism is number 9 in the table in \S 9 of
\cite{Vinberg1976}. Let $\wg^\cC= \ssl(9,\C)$ and fix an isomorphism
$\psi : \wg^\cC \to \g_0^\cC$. 
Since $\SL(9,\C)$ is simply connected the isomorphism $\psi$ lifts to a
surjective morphism $\psi : \wG=\SL(9,\C)\to G_0$. By comparing highest weights
(the highest weight of the module $\g_1^\cC$ is also given in Vinberg's table)
we see that the $\ssl(9,\C)$-module $\g_1^\cC$ is isomorphic to
$V^\cC=\wedge^3 \C^9$
(where $\C^9$ is the natural $\ssl(9,\C)$-module). So the same holds for
the $SL(9,\C)$-modules $\g_1^\cC$ and $\wedge^3 \C^9$. The orbits of $\wG$
acting on $V^\cC$ have been classified in \cite{VE1978}. 
\end{example}

\begin{example}\label{exa:63}
Let $\g^\cC$ be the simple complex Lie algebra of type $E_7$. There is
an automorphism $\theta$ of $\g^\cC$ of order 3 such that $\g_0^\cC$ is
semisimple of type $A_2+A_5$.  This automorphism is number 5 in the table in
\S 9 of \cite{Vinberg1976}. Let $\wg^\cC= \ssl(3,\C)+\ssl(6,\C)$ and fix an
isomorphism $\psi : \wg^\cC \to \g_0^\cC$. Again this
lifts to a surjecive morphism $\psi : \wG=\SL(6,\C)\times \SL(3,\C) \to G_0$.
Set $V^\cC = \wedge^2 \C^6 \otimes \C^2$. By comparing highest weights we
see that the $\wG$-module $\g_1^\cC$ is isomorhic to the $\wG$-module $V^\cC$.
The orbits of $\wG$ acting on $V^\cC$ have been classified in \cite{gatim}.
\end{example}

\subsection{Constructing real representations}

Here we indicate how the setup of the previous subsection can be used to
construct a real Lie group acting on a real vector space. 

We suppose that we have an antiregular involution $\sigma$ of $G$.
This means that $\sigma$ is an involution of $G$ such that for any regular
map $f\in \C[G]$ the map $\bar \sigma^*f$ given by
$$\bar\sigma^*f(g) = \overline{f(\sigma(g))}$$
is regular as well (cf. \cite[Definition 1.7.4]{gw}, \cite[\S 2]{bo21}).
One way of constructing such a map is by choosing a
real form $\g$ of $\g^\cC$. By fixing a basis of $\g$ and considering the
matrix of an element of $G$ with respect to that basis we obtain an
injective homomorphism $\phi : G\to \GL(n,\C)$, where $n=\dim \g^\cC$.
For $x\in \g^\cC$ there are unique $x_1,x_2\in \g$ such that $x=x_1+ix_2$
(where $i\in \C$ is the imaginary unit). Abusing the notation a bit we define
$\sigma : \g\to \g$
by $\sigma(x)=x_1-ix_2$. Then $\sigma$ is an antiinvolution:
$\sigma([x,y])=[\sigma(x),\sigma(y)]$, $\sigma(\lambda x) = \bar\lambda
\sigma(x)$ for $\lambda\in \C$ and $\sigma^2(x)=x$. Let
$x\in \g$ be nilpotent, and set $\sigma(\exp(\ad x)) = \exp( \ad \sigma(x))$.
Then
$$\phi(\sigma(\exp(\ad x))) = \overline{ \phi(\exp(\ad x))}$$
(where the last operation is complex conjugation of the matrix entries in
$\GL(n,\C)$). It follows that $\phi(G)$ is closed under complex conjugation.
So for $g\in G$ we can define $\sigma(g) = \phi^{-1} (\overline{\phi(g)})$.
Then it is clear that $\sigma$ is an antiregular involution of $G$. 

{\em We suppose that $\sigma$ leaves the spaces $\g_i^\cC$ invariant}.
Then we get a real form $\g_0=(\g_0^\cC)^\sigma
=\{ x\in \g_0^\cC \mid \sigma(x) = x \}$ acting on $\g_1=(\g_1^\cC)^\sigma$
(which is defined similarly).
This also implies that $\sigma$ leaves $G_0$ invariant, and
$G_0(\R)=G_0^\sigma = \{ g\in G \mid \sigma(g) = g\}$ is a real Lie group with
Lie algebra $\g_0$. Furthermore, $G_0(\R)$ acts on $\g_1$.
Let $\Gamma = \mathrm{Gal}(\C/\R)=\{1,\gamma\}$. Then $\Gamma$ acts on
the various objects discussed here via $\sigma$. For $g\in G$, $x\in \g^\cC$ we
write $\upgam g = \sigma(g)$, $\upgam x = \sigma(x)$. 

We also suppose that we have an antiregular involution $\sigma$ of
$\wG$ and a conjugation $\sigma$ of $V^\cC$ by which $\Gamma$ acts on $\wG$,
$\wg^\cC$ and $V^\cC$. We suppose that the maps $\wG\to G_0$, $\wg^\cC\to \g_0^\cC$,
$\g_1^\cC\to V^\cC$ are $\Gamma$-equivariant, and that $\sigma(g)\cdot \sigma(v) =
\sigma(g\cdot v)$ for $g\in \wG$, $v\in V^\cC$. 
We want to classify the orbits
of $\wG(\R)=\wG^\sigma$ in $V=(V^\cC)^\sigma$, given those of $\wG$ in $V^\cC$.
This is not equivalent
to classifying the $G_0(\R)$-orbits in $\g_1$ because
$\psi : \wG(\R)\to G_0(\R)$ may not be surjective.

\begin{example}\label{exa:realtrivec}
Let the setup be as in Example \ref{exa:trivec}. By explicitly constructing
the automorphism $\theta$ it can be seen that there is a Chevalley basis of
$\g^\cC$ such that the graded components $\g^\cC_i$ are spanned by elements of
this basis. We let $\g$ be the real form of $\g^\cC$ that is the real span of
this basis. Then the spaces $\g^\cC_i$ are invariant under
the conjugation $\sigma$ of $\g^\cC$. For $i\in \Z_3$ we set $\g_i =
(\g_i^\cC)^\sigma = \{x\in \g_i^\cC \mid \sigma(x)=x\}$. 

On $\SL(9,\C)$ we consider the antiregular map $\sigma$ which is given by the
complex conjugation of the matrix entries of an element of $\SL(9,\C)$.
Its differential $\sigma : \ssl(9,\C)\to \ssl(9,\C)$ is again given by the
complex conjugation of matrix entries. We assume, as we may, that the
map $\psi : \ssl(9,\C)\to \g_0^\cC$
maps $\ssl(9,\R)$ onto $\g_0$. Hence it is $\Gamma$-equivariant. It follows
that also the map $\psi : \SL(9,\C)\to G_0$ is $\Gamma$-equivariant.
So we obtain a representation of $\SL(9,\R)$ on the space $\g_1$. This
representation is equivalent to the natural representation of $\SL(9,\R)$ on
the space $\wedge^3 \R^9$. In this case the map $\psi : \SL(9,\R) \to
G_0(\R)$ is an isomorphism by \cite[Corollary 3.3.14]{bgl21}.


\end{example}

\begin{example}\label{exa:real63}
Let $\g^\cC$, $\theta$ be as in Example \ref{exa:63}. With the same arguments
as in the previous example we construct a representation of $\SL(6,\R)\times
\SL(3,\R)$ which is equivalent to the natural representation of this group
on $\wedge^2 \R^6\otimes \R^2$.
\end{example}  

\subsection{Galois cohomology}

Here we give a very brief introduction to the parts of Galois cohomology that
we need. For more information we refer to \cite{Serre1997}. Here we just
consider Galois cohomology relative to the Galois group $\Gal(\C/\R)$.
This makes it possible to simplify the basic definitions somewhat. However,
the resulting concepts are the same as the usual ones.

Let $\Gamma= \Gal(\C/\R) = \{ 1,\gamma\}$, with $\gamma^2=1$. Let $\Gg$ be
a group on which $\Gamma$ acts by automorphisms. Let $X$ be a set on which
$\Gamma$ acts. For $g\in \Gg$, $x\in X$ we denote their images under
$\gamma$ by $\upgam g$, $\upgam x$. We also suppose that $\Gg$ acts on
$X$ and that this action is $\Gamma$-equivariant: $\upgam g \cdot \upgam x =
\upgam (g\cdot x)$ for all $g\in \Gg$, $x\in X$. 

An element of $\Gg$ is called a {\em cocycle} if $g\upgam g=1$.
Two cocycles $g_1,g_2$ are equivalent if there is an $h\in \Gg$
with $g_1 = h^{-1} g_2 \upgam h$. The first cohomology set $\Ho^1 \Gg$ is the
set of equivalence classes of cocycles. For a cocycle $g$ we denote its
equivalence class by $[g]$ which thus lies in $\Ho^1 \Gg$.

Let $x_0\in X$ be such that $\upgam x_0 = x_0$. Let $Y = \Gg\cdot x_0$ be its
orbit. Let $Z_{x_0} = \{ g\in \Gg\mid g\cdot x_0=x_0\}$ be the stabilizer of
$x_0$. We have a natural map $i_* : \Ho^1 Z_{x_0} \to \Ho^1 \Gg$ mapping the
class of a cocycle $z$ in $Z_{x_0}$ to its class in $\Ho^1 \Gg$. The
kernel of this map consists of all elements that are sent to the trivial
class, so
$$\ker i_* = \{ [z]\in \Ho^1 \Gg \mid z\text{ is equivalent to } 1 \text{ in }
\Gg\}.$$

By $\Gg^\Gamma$, $Y^\Gamma$ we denote the elements of $\Gg$, $Y$ respectively
that are fixed under $\gamma$. We now state the main theorem that we need.
For a proof see \cite[Section I.5.4, Corollary 1 of Proposition 36]{Serre1997}.

\begin{theorem}\label{thm:cohom}
The orbits of $\Gg^\Gamma$ on $Y^\Gamma$ are in bijection with $\ker i_*$.
This bijection is given as follows. Let $[z]\in \ker i_*$, then there is a
$g\in \Gg$ with $z=g^{-1} \upgam g$ and $[z]$ corresponds to the orbit of
$g\cdot x_0$. 
\end{theorem}

\section{Stabilizers of semisimple elements}\label{sec:fams}

Let $(\g^\cC,\theta)$ be a $\Z_m$-graded complex semisimple Lie algebra.
A {\em Cartan subspace} in $\g_1^\cC$ is a maximal subspace consisting of
commuting semisimple
elements. Vinberg (\cite[Theorem 1]{Vinberg1976}) showed that any two Cartan
subspaces of $\g_1^\cC$ are $G_0$-conjugate.

Fix a Cartan subspace $\h^\cC$ in $\g_1^\cC$, and define
\begin{align*}
\Nm_{G_0}(\h^\cC) &= \{ g\in G_0 \mid g\cdot x \in \h^\cC \text{ for all } x\in \h^\cC\},\\
\Zm_{G_0}(\h^\cC) &= \{ g\in G_0 \mid g\cdot x =x \text{ for all } x\in \h^\cC\},\\
W(\g^\cC,\theta) &= \Nm_{G_0}(\h^\cC)/\Zm_{G_0}(\h^\cC).
\end{align*}

The group $W(\g^\cC,\theta)$ is called the {\em Weyl group} of the graded Lie
algebra $(\g^\cC,\theta)$. Since two Cartan subspaces are $G_0$-conjugate,
the Weyl group does not depend (up to isomorphism) on the choice of Cartan
subspace. When $\g^\cC$ and $\theta$ are clear from the context we also
write $W$ instead of $W(\g^\cC,\theta)$. 

Vinberg \cite[Theorem 8]{Vinberg1976} showed that $W$ is generated by
complex reflections. These are invertible linear maps $w$ of $\h^\cC$ of
finite order such that $H_w=\{ p\in \h^\cC \mid wp = p \}$ is of codimension 1.
In this case we say that $H_w$ is a {\em reflection hyperplane} in $\h^\cC$.

By $R(W)$ we
denote the set of all complex reflections in $W$. By
\cite[Lemmas 1.3, 1.6]{lehta} we have that if $H_u=H_v$ for $u,v\in R(W)$
then $1-u$ and $1-v$ have the same image. Hence for a fixed $w\in R(W)$, the
subgroup of $W$ generated by all $u\in R(W)$ such that $H_u=H_w$ is cyclic.

\subsection{Two sets of hyperplanes}

Let $\Sigma(\h^\cC)$ denote the set of the nonzero $\sigma$ in the dual space of
$\h^\cC$ such that there are nonzero $x\in \g^\cC$ with $[h,x] = \sigma(h) x$
for all $h\in \h^\cC$. Consider a Cartan subalgebra of $\g^\cC$ containing
$\h^\cC$ and the corresponding root system. Then the elements of $\Sigma(\h^\cC)$
are the restrictions of the roots in this root system to $\h^\cC$. For this
reason the set  $\Sigma(\h^\cC)$ is called the set of
{\em restricted roots} of $\h^\cC$.

For each $\sigma\in \Sigma(\h^\cC)$ we define the hyperplane
$H_\sigma =\{ h\in \h^\cC \mid \sigma(h)=0\}$. We say that $H_\sigma$ is a
{\em root hyperplane} in $\h^\cC$.

\begin{theorem}\label{thm:hyperplanes}
The set of root hyperplanes in $\h^\cC$ is equal to the set of reflection
hyperplanes in $\h^\cC$.  
\end{theorem}  

\begin{proof}
The proof of the theorem is based on a case by case analysis for simple
$\g$.

First we look at $\g^\cC$ of exceptional type. In this
case all gradings of positive rank (that is, so that the Cartan subspace is
nonzero) are known, see \cite{rlyg}, \cite{levy13}. We can 
skip the gradings of rank 1, as the theorem holds trivially in that case.
For the remaining cases we compute a Cartan subspace $\h^\cC$ and 
the Weyl group $W$ by the algorithm of
\cite{go}. From $W$ we immediately find all its reflections and corresponding
reflection hyperplanes. We also compute a Cartan subalgebra containing
$\h^\cC$, the corresponding root system, the set of restricted roots and
the corresponding root hyperplanes. We have
carried out these computations in the {\sc Magma} computational algebra
system, \cite{magma}. In all cases it turns out that Theorem
\ref{thm:hyperplanes} holds.

For $\g$ of classical type the Weyl groups have been determined in
\cite{Vinberg1976}, see also \cite{levy}. In all cases the Weyl group
is equal to a complex reflection group denoted $G(m,p,r)$. This is the
group consisting of all monomial $r\times r$-matrices whose nonzero entries
$x_1,\ldots,x_r$ satisfy $x_i^m=1$ and $(x_1\cdots x_r)^{\tfrac{m}{p}}=1$. 

Let $c_1,\ldots,c_r$ denote the standard basis of $\C^r$.
In \cite[Lemma 2.8]{lehta} the reflections in $G(m,p,r)$ are characterized.
These are divided into two groups. The first group consists of the
matrices $g$ such that $gc_j= c_j$ for $j\neq i$ and $gc_i = \zeta c_i$ where
$\zeta$ is an $\tfrac{m}{p}$-th root of unity. The second group consists of
the matrices $g$ such that $gc_k=c_k$ for $k\neq i,j$ and $gc_i = \epsilon c_j$,
$gc_j = \epsilon^{-1}c_i$ where $\epsilon$ is an $m$-th root of unity. It follows
that the reflection hyperplanes are exactly the hyperplanes spanned by
$\{c_k \mid k\neq i \}$ for $1\leq i\leq r$ and
$\{c_k\mid k\neq i,j \}\cup \{c_i+\epsilon c_j\}$.
(In particular, the number of reflection hyperplanes is $r+m{r\choose 2}$.)

Now we distinguish a few cases.

{\bf Case 1: $\g = \ssl(n,\C)$ and $\theta$ an inner automorphism of order $m$.}
We consider the Cartan subalgebra of $\g$ consisting of all traceless
diagonal matrices. We may assume that this Cartan subalgebra is $\theta$-stable
and contains a Cartan subspace of $\g_1^\cC$, see, e.g., \cite[\S 3]{Vinberg1976}. Then $\theta$
induces an element of the Weyl group of the Cartan subalgebra (which is the
symmetric group of degree $n$).
Levy (\cite[Lemma 4.5]{levy}) showed that we may assume that this element is
$$(1,\ldots,m)(m+1,\ldots,2m)\cdots ((r-1)m+1,\ldots,rm).$$
Then a Cartan subspace $\h^\cC$ is spanned by $c_1,\ldots,c_r$ where
$c_i$ is the diagonal matrix with $c_i(j,j) = \omega^{-j}$ for
$(i-1)m < j\leq im$ and $c_i(j,j)=0$ otherwise (\cite[Lemma 4.5]{levy}). 
The positive roots of $\g$ are $\alpha_{ij}$ with
$i<j$ where $\alpha_{ij}(\diag(a_1,\ldots,a_n)) = a_i-a_j$. This means that
the root hyperplanes of $\h^\cC$ consist of
those spanned by $c_k$ for $k\neq l$, where $1\leq l\leq m$, along with
those spanned by $\omega^{-j}c_k+\omega^{-i}c_l$ and $c_s$ for $s\neq k,l$, where
$1\leq k < l \leq m$. The Weyl group $W$ in this case is the complex reflection
group $G(m,1,r)$, \cite[Lemma 4.5]{levy}. By the above characterization of the
reflection hyperplanes of this group it follows that the root hyperplanes
coicinde with the reflection hyperplanes.

{\bf Case 2: $\g = \ssl(n,\C)$ and $\theta$ an outer automorphism of order $m$.}
For this we use
the descriptions of Levy in \cite[\S 4.6]{levy}. Here the order $m$ of
$\theta$ is even. We have two cases. In the first case $\tfrac{m}{2}$ is even.
Then we may assume that $\theta$ is such that a Cartan subspace is spanned
by $c_1,\ldots,c_r$ where $c_i$ is the diagonal matrix with $c_i(j,j) =
(-\omega)^{-j}$ for $(i-1)m+1\leq j\leq im$ and $c_i(j,j)=0$ otherwise.
The hyperplanes corresponding to the restricted roots are spanned by
$\{ c_s \mid s\neq i\}$ for $1\leq i\leq r$ and $\{c_s\mid s\neq k,l\}\cup
\{(-\omega)^{-j}c_k+(-\omega)^{-i}c_l\}$. In this case the complex numbers
$(-\omega)^k$ for $1\leq k\leq m$ are distict. So in the second group of
hyperplanes we have $m{r\choose 2}$ distinct hyperplanes spanned by
$\{c_s\mid s\neq k,l\}\cup \{c_k+(-\omega)^{-i}c_l\}$ for $1\leq k<l\leq r$ and
$0\leq i\leq m-1$. Here we have $W=G(m,p,r)$ with $p=1$ or
$p=2$. Hence also in this case the root hyperplanes and reflection hyperplanes
coicide.

If $\tfrac{m}{2}$ is odd then there can be two types of Cartan subspace.
The first type has the same description as in the case where $\tfrac{m}{2}$
is even. So in that case we get the same root hyperplanes. However, as in this
case $(-\omega)^{\tfrac{m}{2}}=1$ we only have the hyperplanes spanned by
$\{c_s\mid s\neq k,l\}\cup\{c_k+(-\omega)^{-i}c_l\}$ for $0\leq i< \tfrac{m}{2}$.
So we get $\tfrac{m}{2}{r\choose 2} +r$ root hyperplanes. The second type
of Cartan subspace is  spanned by $c_1,\ldots,c_r$ where $c_i$ is the diagonal
matrix with $c_i(j,j) =(-\omega)^{-j}$ for $(i-1)m'+1\leq j\leq im'$
and $c_i(j,j)=0$ otherwise, where $m'=\tfrac{m}{2}$. It is straightforward to
see that also in this type the root hyperplanes have the same descriptions.
As here $W=G(\tfrac{m}{2},1,r)$ (\cite[Lemma 4.19]{levy}) we have
verified the theorem also in this case.

{\bf Case 3: $\g = \sp(2n,\C)$ and $\theta$ an inner automorphism of order $m$.}
Now we consider the simple Lie algebras $\sp(2n,\C)$. We use the following
construction from \cite{levy}. Consider the $n\times n$-matrix
$$J_n = \begin{pmatrix} & & 1 \\ & \iddots & \\ 1 & & \end{pmatrix}$$
and the $2n\times 2n$-matrix $M=\begin{pmatrix} 0 & J_n\\ -J_n & 0
\end{pmatrix}$. Then
$$\sp(2n,\C) = \{ x\in \ssl(2n,\C) \mid Mx^t = -Mx\}.$$
We denote the $2n\times 2n$-matrix with a 1 on position $(i,j)$ and zeros
elsewhere by $e_{i,j}$. 
A Cartan subalgebra $\ttt$ of $\sp(2n,\C)$ is spanned by $h_1,\ldots,h_n$ where
$h_k = e_{k,k} -e_{2n+1-k,2n+1-k}$. For a root $\alpha$ we set
$K_\alpha = \{ h\in \ttt \mid \alpha(h)=0\}$, which is a subspace of codimension
1 of $\ttt$. It is straightforward to compute the root vectors and hence the
roots. It follows that the $K_\alpha$ have bases $\{h_k \mid k\neq i\}$,
$\{h_k \mid k\neq i,j\}\cup\{h_i+h_j\}$, $\{h_k \mid k\neq i,j\}\cup\{h_i-h_j\}$,
where $1\leq i< j\leq n$. The Weyl group is the semidirect product of
$S_n$ (permuting the $h_k$) and $\mu_2^n$ (where $\mu_2=\{1,-1\}$ and
$(\epsilon_1,\ldots,\epsilon_n)\cdot h_i = \epsilon_i h_i$). Let $\theta$
be an automorphism of $\sp(2n,\C)$ preserving $\ttt$. Then the restriction
of $\theta$ to $\ttt$ is an element $w$ of the Weyl group. We distinguish
two cases.

In the first case $m$ is odd. Then \cite[Lemma 4.9]{levy} says that
$W=G(2m,1,r)$. Hence the number of hyperplanes corresponding to the reflections
in $W$ is equal to $r+2m{r\choose 2}$. The proof of the cited lemma
shows that we may assume that $w$ is a product of $r$ disjoint $m$-cycles
exactly as in the case for $\ssl(n,\C)$, i.e.,
$$w=(1,\ldots,m)(m+1,\ldots,2m)\cdots ((r-1)m+1,\ldots,rm).$$
The Cartan subspace $\h^\cC$ with basis $c_1,\ldots,c_r$ is also
defined similarly. This means that
$$c_i = \sum_{j=(i-1)m+1}^{im} \omega^{-j} h_j.$$
The hyperplanes corresponding to the
restricted roots of $\h^\cC$ are the intersections of the hyperplanes
$K_\alpha$ and $\h^\cC$.  So we get the hyperplanes spanned by $\{c_k\mid
k\neq i\}$ for $1\leq i\leq r$, $\{c_k\mid k\neq i,j\}\cup\{ \omega^l c_i
-c_j\}$ for $1\leq i < j \leq r$ and $0\leq l\leq m-1$ and $\{c_k\mid k\neq
i,j\}\cup\{ \omega^l c_i+c_j\}$ for $1\leq i < j \leq r$ and $0\leq l\leq m-1$.
Note that $-\omega$ is a primitive $2m$-th root of unity. So the hyperplanes
of the last two groups above are spanned by  $\{c_k\mid k\neq
i,j\}\cup\{ c_i+ (-\omega)^lc_j\}$ for $1\leq i < j \leq r$ and $0\leq l\leq
2m-1$. Again we see that the reflection hyperplanes equal the root hyperplanes.

In the second case $m$ is even. By \cite[Lemma 4.9]{levy} we have
$W=G(m,1,r)$. 
The proof of the cited lemma says that for $w$ we have
two possibilities. The first of these has $w$ equal to a product of $r$
$m$-cycles, as above. However, in this case we have $\omega^{m/2}=-1$.
This implies that the set of hyperplanes spanned by  $\{c_k\mid k\neq
i,j\}\cup\{ \omega^l c_i+c_j\}$ is equal to the set of hyperplanes spanned
by  $\{c_k\mid k\neq i,j\}\cup\{ \omega^l c_i-c_j\}$. So the set of root
hyperplanes is equal to the set of reflection hyperplanes.
The second possibility is that $w$ is a
product of $r$ negative $\tfrac{m}{2}$-cycles. Here a negative
$\tfrac{m}{2}$-cycle $(i_1,\ldots,i_{\tfrac{m}{2}})$ maps
$$h_{i_1} \to h_{i_2}\to \cdots h_{i_{\tfrac{m}{2}}} \to -h_{i_1}.$$
We may assume that $w=(1,\ldots,\tfrac{m}{2})(\tfrac{m}{2}+1,\ldots,
m)\cdots ((r-1)\tfrac{m}{2}+1,\ldots,r\tfrac{m}{2})$.
Since $\omega^{m/2}=-1$ we get a Cartan subspace $\h^\cC$ spanned by
$c_1,\ldots,c_r$ with 
$$c_i = \sum_{j=(i-1)\tfrac{m}{2}+1}^{i\tfrac{m}{2}} \omega^{-j} h_j.$$
Therefore the set of root hyperplanes has the same description as in the
case where $w$ is a product of $r$ $m$-cycles. So we also get the same
conclusion.

{\bf Case 4: $\g = \so(2n+1,\C)$ and $\theta$ an inner automorphism of order
  $m$.}
As in \cite{levy} we use the following construction
$$\so(2n+1,\C) = \{ x\in \ssl(2n+1,\C) \mid J_{2n+1}x^t = -xJ_{2n+2}\}.$$
A basis of this Lie algebra consists of the elements $x_{i,j}=e_{i,j} -
e_{2n+2-j,2n+2-i}$ for $1\leq i\leq 2n+1$, $1\leq j\leq 2n+1-i$. A Cartan
subalgebra $\ttt$ is spanned by $h_1,\ldots,h_n$ where $h_k = x_{k,k}$ for
$1\leq k\leq n$. The $x_{i,j}$ with $i\neq j$ are root vectors. 
From this it is straightforward to determine the roots and the corresponding
hyperplanes $K_\alpha$, which turn out to have the same description as for
$\sp(2n,\C)$. By \cite[Lemma 4.11]{levy} we have that if $m$ is odd then
$W=G(2m,1,r)$ and $w$ may be assumed to be a product of $r$ disjoint
$m$-cycles. If $m$ is even, by the same lemma, $W=G(m,1,r)$ and
we may suppose that $w$ is a product of $r$ negative $\tfrac{m}{2}$-cycles.
In both cases we obtain the desired conclusion in exactly the same way as
for $\sp(2n,\C)$.

{\bf Case 5: $\g = \so(2n,\C)$ and $\theta$ an automorphism of order $m$.}
For the algebras of type $D_n$ we use
$$\so(2n,\C) = \{ x\in \ssl(2n,\C)\mid J_{2n}x^t = -xJ_{2n}\}.$$
A basis of this Lie algebra consists of the elements $x_{i,j}=e_{i,j} -
e_{2n+1-j,2n+1-i}$ for $1\leq i\leq 2n-1$, $1\leq j\leq 2n-i$. A Cartan
subalgebra $\ttt$ is spanned by $h_1,\ldots,h_n$ where $h_k = x_{k,k}$ for
$1\leq k\leq n$. The $x_{i,j}$ with $i\neq j$ are root vectors. 
From this it is straightforward to determine the roots and the corresponding
hyperplanes $K_\alpha$ which are spanned by 
$\{h_k \mid k\neq i,j\}\cup\{h_i+h_j\}$, $\{h_k \mid k\neq i,j\}\cup\{h_i-h_j\}$
for $1\leq i<j\leq n$. {\em If $m$ is odd} then we may assume that $w$ is a
product
of $r$ positive $m$-cycles. Furthermore, by \cite[Lemma 4.13]{levy} we have
that $W=G(2m,p,r)$ with $p=1$ or $p=2$. We have the same set of reflection
hyperplanes in both cases.
{\em If $m$ is even} then we have $W=G(m,p,r)$ with $p=1$ or $p=2$ and we may
assume
that $w$ is a product of $r$ positive $m$-cycles, $r$ negative
$\tfrac{m}{2}$-cycles or  $r$ negative $\tfrac{m}{2}$-cycles and one negative
$1$-cycle.  In all cases we get the desired conclusion in 
exactly the same way as for $\sp(2n,\C)$.
\end{proof}

\begin{example}
We consider the automorphism $\theta$ of order 3 of the simple Lie algebra of
type $E_8$ from Example \ref{exa:trivec}. In \cite{VE1978}, Vinberg and
Elashvili explicitly described a Cartan subspace $\h^\cC$ in $\g_1^\cC$.
Furthermore, they explicitly constructed the generating reflections of the
Weyl group. Here we summarize this construction. It shows that Theorem
\ref{thm:hyperplanes} holds in this case.

Let $\ttt^\cC$ denote the centralizer of $\h^\cC$ in $\g^\cC$. It is a Cartan
subalgebra of $\g^\cC$. It is invariant under $\theta$. Denote by $\theta_*$ the
linear transformation of the dual space $(\ttt^\cC)^*$ given by
$(\theta_*(\alpha))(h) = \alpha(\theta(h))$ for $h\in \ttt^\cC$. In
\cite{VE1978} it is shown that for a root $\alpha$ relative to $\ttt^\cC$
the elements $\pm \alpha$, $\pm \theta_*(\alpha)$, $\pm \theta_*^2(\alpha)$
form a root subsystem of type $A_2$. Let $\uuu(\alpha)$ denote the regular
subalgebra of $\g^\cC$ with this root system. Let $U(\alpha)\subset
G$ be the subgroup corresponding to this subalgebra. The restriction of
$\theta$ to $\uuu(\alpha)$ has to be an inner automorphism. Hence there is
a $g_\alpha \in U(\alpha)$ such that $g_\alpha(x) = \theta(x)$ for
$x\in \uuu(\alpha)$. We have that $g_\alpha \in G_0$.

We have a decomposition $\ttt^\cC = \ttt(\alpha)\oplus \ttt_0(\alpha)$,
where $\ttt(\alpha) = \ttt\cap \uuu(\alpha)$ is of dimension 2 and
$\ttt_0(\alpha) = \{ x\in \ttt^\cC \mid \alpha(x) = \theta_*(\alpha)(x)=0\}$.
The Cartan subspace decomposes accordingly, $\h^\cC = \h(\alpha)\oplus
\h_0(\alpha)$ where $\h(\alpha) = \ttt(\alpha)\cap \h^\cC$ and likewise for
$\h_0(\alpha)$. We have that $\dim \h(\alpha)=1$.
The element $g_\alpha$ stabilizes $\h^\cC$ hence its restriction $w_\alpha$
to $\h^\cC$ belongs to the Weyl group $W$. Furthermore
$w_\alpha(x) = \omega x$ for $x\in \h(\alpha)$ and $w_\alpha(x)=x$ for $x\in
\h_0(\alpha)$. It follows that $w_\alpha$ is a complex reflection.

The 240 roots of $\g^\cC$ with respect to $\ttt^\cC$ are partitioned into
40 subsets of type $A_2$ as above. Hence we obtain 40 reflections $w_\alpha$.
Together with their squares they exhaust all complex reflections in $W$.
All roots in the same $A_2$ subsystem give the same root hyperplane $\h_0(
\alpha)$. It follows that the root hyperplanes coincide with the reflection
hyperplanes.
\end{example}

\subsection{Two consequences}

In this section we derive two interesting consequences of Theorem
\ref{thm:hyperplanes}. These are Theorems \ref{thm:circreg} and \ref{thm:fam}.

For $p\in \h^\cC$ let
$$W_p =\{ w\in W \mid w\cdot p = p\}$$
denote its stabilizer in $W$. A subgroup of $W$ is said to be a
{\em reflection subgroup} if it is generated by elements of $R(W)$. 
By a theorem of Steinberg, \cite[Theorem 1.5]{steinberg}, $W_p$ is a reflection
subgroup of $W$. Hence it is generated by all $w\in R(W)$ such that $p\in H_w$.

For $p\in \h^\cC$ we define
\begin{align*}
\h_p^\cC &= \{ q\in \h^\cC \mid w\cdot q = q \text{ for all } w \in W_p\} =
\{ q\in \h^\cC \mid W_p \subset W_q\},\\
\h_p^{\cC,\circ} &= \{ q\in \h^\cC\mid W_q= W_p\}.
\end{align*}

Let $w_1,\ldots, w_k$ be a set of complex reflections generating $W_p$.
Then $\h_p^\cC$ is the intersection of the $H_{w_i}$ for $1\leq i\leq k$.
So $\h_p^\cC$ is a subspace of $\h$. Furthermore, $q\in \h_p^{\cC,\circ}$ if and
only if
$q\in\h_p^\cC$ and $(w-1)q\neq 0$ for all complex reflections $w\in R(W)$
that do not lie in $W_p$. So $\h_p^{\cC,\circ}$ is Zariski-open in $\h_p^\cC$.
 
There is a finite number of reflection subgroups in $W$. Hence there is also a
finite number of sets $\h_p^{\cC,\circ}$, and $\h^\cC$ is their disjoint union.

It is straightforward to see that for $p,p'\in \h^\cC$ and $w\in W(\g^\cC,\theta)$ we have
\begin{equation}\label{eq:CpCq}
w\h_p^{\cC,\circ} = \h_{p'}^{\cC,\circ} \text{ if and only if } wW_pw^{-1} = W_{p'}.
\end{equation}  

Indeed, if $w\h_p^{\cC,\circ} = \h_{p'}^{\cC,\circ}$ then in particular
$wp\in \h_{p'}^{\cC,\circ}$
so that $W_{wp} = W_{p'}$, which is the same as $wW_pw^{-1} = W_{p'}$. For the
converse let $q\in \h_p^{\cC,\circ}$; then $W_{wq} = wW_qw^{-1} = wW_pw^{-1} = W_{p'}$,
so that $wq\in \h_{p'}^{\cC,\circ}$. We conclude $w\h_p^{\cC,\circ} \subset
\h_{p'}^{\cC,\circ}$.
In the same way we see that $w^{-1} \h_{p'}^{\cC,\circ}\subset \h_p^{\cC,\circ}$
which implies
that $w\h_p^{\cC,\circ} = \h_{p'}^{\cC,\circ}$.

For a subspace $\a^\cC\subset \h^\cC$ let
$$\Sigma(\a^\cC)= \{ \sigma \in \Sigma(\h^\cC) \mid \text{ there is a }
p\in \a^\cC \text{ such that } \sigma(p)\neq 0\}.$$
Then for $p\in \h^\cC$ we define
$$\h_p^{\cC,\reg}=\{ q\in \h_p^\cC \mid \sigma(q)\neq 0 \text{ for all }
\sigma\in \Sigma(\h_p^\cC)\}.$$

\begin{theorem}\label{thm:circreg}
Let $r\in \h^\cC$ then $\h_r^{\cC,\circ} = \h_r^{\cC,\reg}$.
\end{theorem}

\begin{proof}
Let $r\in \h^\cC$ and let $q\in \h_r^{\cC,\circ}$;
we show that $q\in \h_r^{\cC,\reg}$. Let $\sigma\in \Sigma(\h_r^\cC)$. By Theorem
\ref{thm:hyperplanes} there is a $w\in R(W)$ such that
$H_w=H_\sigma$. There is a $p\in \h_r^{\cC}$ with $\sigma(p)\neq 0$. This implies
$w(p)\neq p$ and hence $w\not\in W_r$ (by definition of $\h_r^\cC$). Therefore
$w\not\in W_q$ so that $q\not\in H_w$ and $\sigma(q)\neq 0$. We conclude
that $q\in \h_r^{\cC,\reg}$.

Conversely, let $q\in \h_r^{\cC,\reg}$. Let $w\in R(W)$, $w\not\in W_r$; we
show that $w\not\in W_q$. By Theorem \ref{thm:hyperplanes}
there is a $\sigma\in \Sigma(\h^\cC)$ such that $H_\sigma=H_w$.
Hence $w(r)\neq r$ is the same as $\sigma(r)\neq 0$. So $\sigma\in
\Sigma(\h_r^{\cC})$. 
This implies that $\sigma(q)\neq 0$ and hence $w(q)\neq q$.
As $W_q$ is a reflection subgroup of $W$ and $W_r\subset W_q$
we obtain that $W_q=W_r$ and $q\in \h_r^{\cC,\circ}$.
\end{proof}  

To formulate the second consequence of Theorem \ref{thm:hyperplanes} we need the
following definition.

\begin{definition}
We say that $p,q\in \h^\cC$ lie in the same {\em Jordan class} 
if the centralizers $\Zgg_{\g^\cC}(p)$, $\Zgg_{\g^\cC}(q)$
\end{definition}

One can define Jordan classes for general elements in $\g_1$, see \cite{ces}.
For semisimple elements the general definition reduces to the one given here.
In \cite{ces} Jordan classes in $\g_1$ are studied in detail. Here we
give a characterization of the Jordan classes in $\h^\cC$ in terms of
stabilizers in the Weyl group.

\begin{lemma}\label{lem:cenconj}
Let $p,q\in \h^\cC$ and suppose that $\Zgg_{\g^\cC}(p)$, $\Zgg_{\g^\cC}(q)$ are
$G_0$-conjugate. Then there exists a $g\in \Nm_{G_0}(\h^\cC)$ such that
$g\cdot \Zgg_{\g^\cC}(p) = \Zgg_{\g^\cC}(q)$. 
\end{lemma}

\begin{proof}
Let $g'\in G_0$ be such that $g'\cdot \Zgg_{\g^\cC}(p) = \Zgg_{\g^\cC}(q)$.
Then $\h^\cC$ and $g'\cdot \h^\cC$ are Cartan subspaces of the $\Z_m$-graded
Lie algebra $(\Zgg_{\g^\cC}(q), \theta)$. Hence by Vinberg's theorem
(\cite[Theorem 1]{Vinberg1976}) there is a $g''\in \Zz_{G_0}(q)$ with
$g'' g' \h^\cC = \h^\cC$. As $g''\in \Zz_{G_0}(q)$ we also have
$g'' \cdot \Zgg_{\g^\cC}(q) = \Zgg_{\g^\cC}(q)$. Hence with $g = g''g'$ we get
the conclusion of the lemma.
\end{proof}

\begin{lemma}\label{lem:gfamWfam}
If $p,q\in \h^\cC$ lie in the same $\g^\cC$-family then the stabilizers
$W_p$, $W_q$ are $W$-conjugate.  
\end{lemma}

\begin{proof}
By the previous lemma there is a $g\in \Nm_{G_0}(\h^\cC)$ such that
$g\cdot \Zgg_{\g^\cC}(p) = \Zgg_{\g^\cC}(q)$. By Steinberg's theorem
(\cite[Corollary 3.11]{steinberg75}) the groups $\Zz_G(p)$, $\Zz_G(q)$ are
connected. The Lie algebra of $g\Zz_G(p)g^{-1}$ is $g\cdot \Zgg_{\g^\cC}(p)$. 
Hence $g\Zz_G(p) g^{-1} = \Zz_G(q)$. As $\Zz_{G_0}(p) =
\Zz_{G}(p)\cap G_0$, and similarly for $\Zz_{G_0}(q)$, we also have
$g\Zz_{G_0}(p) g^{-1} = \Zz_{G_0}(q)$. Let $w\in W$ be the element induced
by $g$. Let $u\in W_p$ be induced by $g_u\in \Nm_{G_0}(\h^\cC)$. Then
$g_u\in \Zz_{G_0}(p)$. Furthermore, $gg_ug^{-1} \in \Zz_{G_0}(q)$ translates
to $wuw^{-1}\in W_q$. It follows that $wW_pw^{-1} \subset W_q$. The other
inclusion is proved similarly. 
\end{proof}

For a subset $S\subset \h^\cC$ we set $\Zgg_{\g^\cC}(S) = \{x\in \g^\cC\mid
[x,s]=0 \text{ for all } s\in S\}$. Also define $\Aa_S = \{ \sigma \in
\Sigma(\h^\cC) \mid \sigma(s)=0 \text{ for all } s\in S\}$. If $S=\{p\}$
then we also write $\Zgg_{\g^c}(p)$ and $\Aa_p$. Now we have the following
elementary fact
\begin{equation}\label{eq:centralizer}
\Zgg_{\g^\cC} (S) = \Zgg_{\g^\cC}(\h^\cC) \oplus \bigoplus_{\sigma\in \Aa_S}
\g^\cC_\sigma.
\end{equation}

The next theorem is stated without proof in \cite{AE1982}. Since we could not
find a proof in the literature we have included a proof in the present paper.

\begin{theorem}\label{thm:fam}
Two elements $p,q\in \h^\cC$ lie in the same Jordan class if and only if
the stabilizers $W_p$, $W_q$ are $W$-conjugate.  
\end{theorem}

\begin{proof}
 We first show that for $r\in \h^\cC$ we have
$\Zgg_{\g^\cC}(r) = \Zgg_{\g^\cC}(\h_r^{\cC,\reg})$. Let $\sigma\in \Sigma(\h^\cC)$.
If $\sigma |_{\h_p^\cC} \neq 0$ then $\sigma(p)\neq 0$. It follows that
$$\{\sigma\in \Sigma(\h^\cC)\mid \sigma(p)=0\} = \{\sigma\in \Sigma(\h^\cC)\mid
\sigma |_{\h_p^{\cC,\reg}}=0\}.$$
This immediately implies the desired conclusion.

By Lemma \ref{lem:gfamWfam} it suffices to show the ``if'' part. 
Let $p,q\in \h^\cC$ and suppose that there is a $w\in W$
with $W_q= wW_pw^{-1}$. By \eqref{eq:CpCq} this implies $w \h_p^{\cC,\circ}
=\h_q^{\cC,\circ}$. By (3) that yields  $w \h_p^{\cC,\reg}=\h_q^{\cC,\reg}$.
Now let $g\in \Nm_{G_0}(\h^\cC)$ induce $w$. Then
$$g\cdot \Zgg_{\g^\cC} (\h_p^{\cC,\reg}) = \Zgg_{\g^\cC} (\h_q^{\cC,\reg}).$$
Hence $g\cdot \Zgg_{\g^\cC} (p) = \Zgg_{\g^\cC} (q)$ and $p,q$ lie in the
same Jordan class. 
\end{proof}

\subsection{Classification of semisimple orbits}\label{sec:clasC}

Let $(\g^\cC,\theta)$ be a $\Z_m$-graded complex semisimple Lie algebra.
Since any two Cartan subspaces of $\g_1^\cC$ are $G_0$-conjugate
(\cite[Theorem 1]{Vinberg1976}), it follows that every semisimple $G_0$-orbit
in $\g_1^\cC$ has a point in any fixed Cartan subspace.

As in the previous section we consider the Weyl group $W(\g^\cC,\theta)$ of
the graded Lie algebra $(\g^\cC,\theta)$.
\cite[Theorem 2]{Vinberg1976} says that two elements of $\h^\cC$ are
$G_0$-conjugate if and only if they are $W(\g^\cC,\theta)$-conjugate. So this
yields
a classification of the semisimple $G_0$-orbits: a set of representatives of
the $W(\g^\cC,\theta)$-orbits on $\h^\cC$ is also a set of representatives of the
semisimple $G_0$-orbits in $\g_1^\cC$. However, using a procedure outlined in
\cite{VE1978} this classification can be refined a bit.

\begin{lemma}\label{lem:redweyl}   
Let $p\in \h^\cC$ and $q_1,q_2 \in \h^{\cC,\circ}_p$. If $q_1=w \cdot q_2$ for a
$w \in W$ then $w \in \Nn_W(W_p)$.
\end{lemma}	

\begin{proof}
Indeed, $W_p= W_{q_1} = W_{wq_2} = wW_{q_2}w^{-1} =wW_pw^{-1}$.
\end{proof}

Write $W=W(\g^\cC,\theta)$. For $p\in \h^\cC$ we set
$$\Gamma_p = \Nn_W(W_p)/W_p.$$

\begin{corollary}\label{cor:Gampconj}
Let $p\in \h^\cC$; then $q_1,q_2\in \h_p^{\cC,\circ}$ are $G_0$-conjugate if and
only if they are $\Gamma_p$-conjugate.
\end{corollary}

There is only a finite number of conjugacy classes of reflection subgroups
of $W(\g^\cC,\theta)$. Not all of these are necessarily of the form $W_p$ for
a $p\in \h^\cC$; however there are $p_1,\ldots,p_N\in \h^\cC$ such that for any
$q\in \h^\cC$ the stabilizer $W_q$ is conjugate to precisely one of the
$W_{p_i}$. Then from \eqref{eq:CpCq}
it follows that every semisimple $G_0$-orbit in $\g_1^\cC$ has a point in exactly
one of the $\h_{p_i}^{\cC,\circ}$. Moreover, the previous corollary says that
two elements of $\h_{p_i}^{\cC,\circ}$ are $G_0$-conjugate if and only if they are
$\Gamma_{p_i}$-conjugate. Hence the union of sets of representatives of the
$\Gamma_{p_i}$-orbits in $\h_{p_i}^{\cC,\circ}$, for $1\leq i\leq N$, is also a set
of representatives of the semisimple $G_0$-orbits in $\g_1^\cC$.

\begin{example}
Let the notation be as in Example \ref{exa:63}. In \cite{gatim} an explicit
Cartan subspace $\h^\cC$ of dimension 3 is given. Furthermore it is shown that
the Weyl group $W$ is isomorphic to the complex reflection group denoted
$G_{26}$ in the Shephard-Todd classification. This group has order 1296.
In \cite{gatim} the spaces
$\h_p^{\cC,\reg}$ up to $W$-conjugacy are classified by direct arguments. Here we
indicate how this can be done by considering the reflection subgroups of $W$.

It is known that $W$ has 16 reflection subgroups up to conjugacy. In the
following table we list their orders along with the number of generating
reflections.

\begin{center}
\begin{tabular}{l|l}
  number of generators & order of subgroups\\
  \hline
0 & 1\\
1 & 2, 3\\
2 & 6, 6, 9, 18, 24\\
3 & 18, 27, 48, 54, 54, 162, 648, 1296
\end{tabular}
\end{center}

Let $H$ be a reflection subgroup and set $\h^\cC_H = \{ q\in \h^\cC \mid g(q)=q
\text{ for all } g\in H\}$. If $H=W_p$ then $\h^\cC_H = \h^\cC_p$. Now let $p\in
\h^\cC_H$. If $W_{p} =H$ then we have found a $p$ such that $H=W_p$. If
$W_p\supsetneq H$ then we consider $\h^\cC_{W_p}$: if this set is equal
to $\h^\cC_H$ then there is no $p'\in \h^\cC$ with $H=W_{p'}$. After trying a few
random $p\in \h^\cC_H$ we end up in one of the two possibilities with high
probability.

In our example we have carried this out with the help of the computational
algebra system {\sf GAP}4, \cite{gap4}. We do not give the explicit form
of the elements that we found, because it would require to also give
all generating reflections, which we want to avoid for reasons of brevity.
Here the computations are very straightforward. If $H$ is the trivial subgroup
then $\h^\cC_H=\h^\cC$ and we quickly find a $p$ such that $H=W_p$. The
1-generated subgroups $H$ both have $\dim \h^\cC_H=2$ and again we quickly
find a suitable $p$. The 2-generated subgroups $H$ all have $\dim \h^\cC_H=1$;
so in these cases it suffices to try any nonzero $p$. It turns out that only
one subgroup of order 6 and the subgroups of order 18 and 24 are of the form
$W_p$ for a $p\in \h^\cC$. The 3-generated subgroups $H$ all have $\dim
\h^\cC_H=0$ so the only subgroup of the form $W_p$ is $W$ itself of order
1296.

If $W_p$ is the trivial subgroup then $\Gamma_p=W$. For the subgroups of
order 2, 3 we have $|\Gamma_p| = 72, 36$ respectively. For all 2-generated
subgroups we have $|\Gamma_p|=6$. We see that we recover exactly the
classification given in \cite{gatim}. 
\end{example}

By the next
corollary all elements of a $\h_{p_i}^{\cC,\circ}$ have the same stabilizer
in $G_0$.

\begin{corollary}\label{cor:cen}
Let $p\in \h^\cC$, and $q_1,q_2\in \h_p^{\cC,\circ}$. Then $\Zz_{G_0}(q_1) =
\Zz_{G_0}(q_2)$ and $\Zz_{\wG}(q_1) = \Zz_{\wG}(q_2)$.
\end{corollary}

\begin{proof}
From Theorem \ref{thm:circreg}, in view of \eqref{eq:centralizer},
it follows that
$\Zgg_{\g^\cC}(q_1) = \Zgg_{\g^\cC}(q_2)$. Since $\Zz_G(q_i)$ is connected for
$i=1,2$ (\cite[Corollary 3.11]{steinberg75}) it follows that
$\Zz_{G}(q_1)=\Zz_{G}(q_2)$. Hence $\Zz_{G_0}(q_1) = G_0\cap \Zz_G(q_1) =
G_0\cap \Zz_G(q_2) = \Zz_{G_0}(q_2)$. Since $\Zz_{\wG}(q_i)$ is the pre-image of
$\Zz_{G_0}(q_i)$ in $\wG$ the second statement follows as well. 
\end{proof}

This corollary makes it possible to classify the mixed elements whose semisimple
part lies in a fixed set $\h_p^{\cC,\circ}$. Indeed, let $q\in \h_p^{\cC,\circ}$ and
consider mixed elements $x=q+e$, where $e$ is nilpotent with $[q,e]=0$.
Set $\a_q^\cC = \Zgg_{\g^\cC}(q)$ and 
consider the graded Lie algebra $(\a_q^\cC,\theta)$. We have that
$e\in \a_{q,1}^\cC$. Moreover, $q+e$, $q+e'$ are $G_0$-conjugate if and only
if $e,e'$ are $\Zz_{G_0}(q)$-conjugate. By Corollary \ref{cor:cen} we have
that $\Zz_{G_0}(q)=\Zz_{G_0}(p)$ and $\a_q^\cC = \a_p^\cC$. So by classifying the
nilpotent elements in $\a_{p,1}^\cC$ under the action of the group $\Zz_{G_0}(p)$
we find a list of nilpotent parts of mixed elements whose semisimple part
lies in $\h_p^{\cC,\circ}$. In other words, one list of semisimple parts
does the job for all semisimple parts from $\h_p^{\cC,\circ}$ at once.

\section{On the classification of real semisimple orbits}\label{sec:realss}

Here we study two questions related to the semisimple orbits of a real form
of a complex $\theta$-representation. The first question concerns the
classification of the real Cartan subspaces. The definition of this concept
in the real case is the same as in the complex case. However, in the real case
it is no longer true that there is just one Cartan subspace up to conjugacy.
The second question asks which complex semisimple orbits have real points.
We show a theorem that characterises those orbits assuming that 
the group $\wG$ has trivial Galois cohomology and also that certain cocycles
can be lifted.

In the sequel, when $U$ is a subspace of $\g$ then by $U^\cC$ we denote the
subspace of $\g^\cC$ spanned by a basis of $U$.

\subsection{Real Cartan subspaces}

Let $(\g^\cC,\theta)$ be a $\Z_m$-graded semisimple Lie algebra. Similarly
to the complex case we define a Cartan
subspace of $\g_1$ to be a maximal subspace consisting of commuting semisimple
elements. One of the main objectives of this section is to show that
$\h\subset \g_1$ is a Cartan subspace if and only if $\h^\cC$ is a Cartan
subspace of $\g_1^\cC$.

Let $U$ be a subspace of $\g$ then by $\Zgg_\g (U)$  we denote the centralizer
of  $U$  in $\g$. Also we set 
\begin{equation}
U_i = U\cap \g_i.  \label{eq:gr1}
\end{equation}

For a Lie algebra $\a$ we let $D(\a)=[\a,\a]$ be its derived subalgebra
and let $Z(\a)$ be its centre.

\begin{lemma}\label{lem:centr1}
Let $\h \subset \g_1$  be a subspace consisting of commuting semisimple
elements. Set $\a=\Zgg_\g (\h)$. 
Then $\a$ is reductive in $\g$, which means that $\a = D(\a)\oplus
Z(\a)$, $D(\a)$ is semisimple and $Z(\a)$ consists of semisimple elements.

Moreover, we have
\begin{equation}
\a = \bigoplus _i \a_i.\label{eq:dec1}
\end{equation}
and
\begin{equation}
D(\a) =  \bigoplus _i D(\a)_i \text{ and } Z(\a) =  \bigoplus _i Z(\a)_i.
\label{eq:dec2}
\end{equation}
\end{lemma}

\begin{proof}
The first assertion is well-known, see for example
\cite[Proposition 20.5.13]{tauyu}.

For $x \in \a$ write $x = \sum _i x_i$, where $x_i \in \g_i$. Let $t\in \h$.
Since  $[x, t ] = 0$ and $[x_i, t ] \in  \g_{i+1}$ it follows
that $[x_i,t]=0$. This proves \eqref{eq:dec1}.
	
Since  $\Zgg_\g(\h)^\cC = \Zgg_{\g ^\cC}(\h^\cC)$ and $\h^\cC$ is invariant under
$\theta$ we see that $\a^\cC$ is invariant under $\theta$.
Furthermore, $D(\a^\cC)$ and $Z(\a^\cC)$ are invariant under $\theta$ as well.
Hence $D(\a^\cC)$ is the direct sum of the spaces $D(\a^\cC)\cap \g_i^\cC$.
So if $x\in D(\a)$ then $x=\sum_i x_i$ with $x_i \in D(\a^\cC)\cap \g_i^\cC$.
As $x\in \g$ we also have $x_i\in \g_i$. Since $D(\a^\cC) = D(\a)^\cC$ we
get $x_i\in D(\a)$, showing the first assertion of \eqref{eq:dec2}. The
second one is proved analogously.
\end{proof}

The following lemma is a version of \cite[Prop. 6]{Vinberg1976}
for  real $\Z_m$-graded  semisimple Lie algebras.

\begin{lemma}\label{lem:carc}
A subspace $\h$ consisting of commuting semisimple elements in $\g_1$ is a
Cartan subspace if and only if $Z(\Zgg_\g(\h))_1 = \h$  and $D(\Zgg_\g(\h))_1$
consists of of nilpotent elements.
\end{lemma}

\begin{proof}
Assume that $\h$ is a Cartan subspace. By Lemma \ref{lem:centr1}
$Z(\Zgg_\g(\h))_1$ consists of commuting semisimple elements. So because
$\h \subset  Z(\Zgg_\g(\h))_1$ we see that $Z(\Zgg_\g(\h))_1 = \h$. 
Now let $x \in D(\Zgg_\g(\h))_1$. We  will show that $x$ is nilpotent.
The grading in $\g$ induces a grading in $D(\Zgg_\g(\h))$; see \eqref{eq:dec2}.
We write the Jordan decomposition $x=s+n$ in the real semisimple Lie algebra
$D(\Zgg_\g(\h))$, where $s\in D(\Zgg_\g(\h))_1$ is semisimple and
$n\in D(\Zgg_\g(\h))_1 $ is nilpotent. In particular, $s\in \Zgg_\g(\h)$
so that $s$ commutes with $\h$. As $s$ is semisimple we conclude that
$s\in\h=Z(\Zgg_\g(\h))_1$. On the other hand $s$ is contained in
$D(\Zgg_\g(\h))_1$. Since $Z(\Zgg_\g(\h))\cap D(\Zgg_\g(\h))=\{0\}$
we conclude that $s=0$ and $x=n$ is nilpotent.

Now we assume that $Z(\Zgg_\g(\h))_1 = \h$ and that $D(\Zgg_\g(\h))_1$ consists
of nilpotent elements. Let  $\t\supseteq \h$ be a subspace consisting of
commuting semisimple elements in $\g_1$. Let $x\in \t$. By Lemma
\ref{lem:centr1} we can write $x=s+n$, where $s\in Z(\Zgg_\g(\h))_1$ and
$n\in D(\Zgg_\g(\h))_1$. Hence $s$ is semisimple, $n$ is nilpotent
and $[s,n]=0$. It follows that this is the Jordan decomposition of $x$.
Since $x$ is semisimple we see that $n=0$. Hence $\t\subset Z(\Zgg_\g (\h))_1$,
so $\t = \h$.
\end{proof}

\begin{corollary}\label{cor:complc}
Let $\h\subset \g_1$ be a subspace consisting of commuting semisimple
elements. Then $\h$ is a Cartan subspace of $\g_1$ if and only if
$\h^\cC$ is a Cartan subspace of $\g_1^\cC$.
\end{corollary}

\begin{proof}
Assume  that  $\h \subset \g_1$ is a Cartan subspace. By Lemma \ref{lem:carc},
$D(\Zgg_{\g}(\h))_1$ consists of nilpotent elements. Set $\a = D(\Zgg_{\g}(\h))$.
Consider the $\Z_m$-graded semisimple Lie algebra $(\a^\cC,\theta)$. Let
$A_0$ be the connected subgroup of $G_0$ with Lie algebra $\a_0^\cC$.
By  \cite[Theorem 3]{Vinberg1976}
the set of nilpotent elements in $\a_1^\cC$ coincides with the null cone
with respect to the action of $A_0$.
In other words, if there are non-nilpotent
elements in $\a_1^\cC$ then there is a nontrivial $A_0$-invariant
$p$ in the polynomial ring $\C[\a_1^\cC]$. But then there also exist
$x\in \a_1$ with $p(x)\neq 0$. Such $x$ are not nilpotent in contradiction
with Lemma \ref{lem:carc}. By \cite[Proposition 6]{Vinberg1976} it now follows
that  $\h^\cC$ is a Cartan subspace in $\g_1^\cC$.

In particular it follows that all Cartan subspaces in $\g_1$ have  the same
dimension, which is the  dimension of any Cartan subspace in $\g_1^\cC$.
Now assume that $\h^\cC$  is a Cartan  subspace. Then  $\h$ consists  of
commuting semisimple elements and  $\dim _\R \h = \dim  _\C \h^\cC$.
Hence $\h$ is a real  Cartan  subspace  in $\g_1$.
This completes the  proof  of Corollary \ref{cor:complc}.
\end{proof}

We recall that the $G_0(\R)$-module $V$ is the real form of $V^\cC$ given by
$V=(V^\cC)^\sigma$. We have a linear isomorphism $V\to \g_1$ which is the
restriction of the isomorphism of $G_0$-modules $V^\cC\to \g_1^\cC$.
We say that a subspace of the module $V$ is a Cartan subspace of its image
in $\g_1$ is a Cartan subspace. Now we consider the problem of classifying the
Cartan subspaces in $V$ up to $\wG(\R)$-conjugacy. As we have been considering
Cartan subspaces in $\g_1$ we continue to do so, identifying $V$ and $\g_1$.

Let  $\h$ be  a  fixed Cartan subspace in $\g_1$. By Corollary
\ref{cor:complc} $\h^\cC$ is a  Cartan  subspace  in $\g^\cC_1$.
We set
$$\Nn_0=\Nn_{\wG}(\h^\cC)=\{ g\in\wG \mid g(\h^\cC) = \h^\cC\}.$$

\begin{theorem}\label{thm:cartan}
There is a bijection between  the  $\wG(\R)$-conjugacy classes  of Cartan
subspaces  in $\g_1$  and the kernel  $\ker[\hs \Ho^1 \Nn_0 \to \Ho^1
\wG\hs]$.
\end{theorem}

\begin{proof}
Let $d$ denote the dimension of a Cartan subspace  in $\g_1^\cC$.
Let $X= \mathrm{Gr}(\g_1^\cC,d)$ be the set of $d$-dimensional subspaces of
$\g_1^\cC$. Then $\Gamma$ acts on $X$, and a subspace lies in $X^\Gamma$ if and
only if it has a basis of elements in $\g_1$. Also $\wG$ acts on $X$, and 
let $Y$ be the orbit of $\h$. Then by Vinberg's theorem, saying that
two Cartan subspaces in $\g_1^\cC$ are $\wG$-conjugate, immediately implies
that $Y$ consists of all Cartan subspaces in $\g_1^\cC$. The stabilizer of
$\h$ is exactly $\Nn_0$. By Corollary \ref{cor:complc} we see that the elements
of $Y^\Gamma$ are exactly the real Cartan subspaces in $\g_1$.
Hence the theorem follows from Theorem \ref{thm:cohom}
\end{proof}

\begin{remark}  Theorem \ref{thm:cartan} has been stated  as \cite[Theorem 4.4.9]{bgl21}
 with a slightly different  proof.
\end{remark}

\begin{example}
Let the setup be as in Example \ref{exa:realtrivec}. In \cite{bgl} it has been
computed that $\Ho^1 \Nn_0$ is trivial. Therefore there is one real
Cartan subspace up to conjugacy. 
\end{example}

\begin{example}
Let the setup be as in Example \ref{exa:real63}. We let $u_1,\ldots,u_6$,
$v_1,v_2,v_3$ be the standard bases of $\C^6$ and $\C^3$ respectively (or of
$\R^6$ and $\R^3$). By $e_{ijk}$ we denote the basis element $u_i\wedge u_j
\otimes v_k$ of $\wedge^2 \C^6\otimes \C^3$ (or of $\wedge^2 \R^6\otimes \R^3$).
It is shown in \cite{gatim} that there is a Cartan subspace $\h^\cC$
spanned by the elements $e_{121}+e_{342}+e_{563}$, $-e_{451}+e_{162}-e_{233}$,
$e_{361}-e_{252}+e_{143}$. Since these elements have real coefficients, it is
also a basis of a real Cartan subspace $\h$.

With explicit computations in {\sf GAP}4 and {\sc Magma} \cite{magma}
(we used the latter for Gr\"obner basis computations) we determined the
group
$$\Zm_0 = \{ g\in \SL(6,\C)\times \SL(3,\C) \mid g(p)=p \text{ for all }
p\in \h^\cC\}.$$
The identity component of this group is a 1-dimensional torus consisting
of the elements $\diag(t,t^{-1},t,t^{-1},t,t^{-1},1,1,1)$ for $t\in \C^*$.
The group has 27 components which we determined explicitly, but do not report
here.

The group $W$ is generated by three reflections. Again by explicit Gr\"obner
basis computations we determined three elements $g_1,g_2,g_3$
of $\SL(6,\C)\times \SL(3,\C)$
normalizing $\h^\cC$ and inducing the three given reflections. Then
$\Nn_0$ is generated by $g_1,g_2,g_3$ along with $\Zm_0$. From this it is possible to
establish that $\Nn_0$ has 34992 components (and the same identity component
as $\Zm_0$).

Since we have the group $\Nn_0$ explicitly it is not difficult to establish
that $\Ho^1 \Nn_0 =\{ [1], [c]\}$. Furthermore, it is a short calculation
to find a $g\in \SL(6,\C)\times \SL(3,\C)$ with $g^{-1} \bar g = c$. We find
$$c= \SmallMatrix{0&0&0&i&0&0&0&0&0\\
  0&0&0&0&i&0&0&0&0\\
  0&0&0&0&0&i&0&0&0\\
  i&0&0&0&0&0&0&0&0\\
  0&i&0&0&0&0&0&0&0\\
  0&0&i&0&0&0&0&0&0\\
  0&0&0&0&0&0&1&0&0\\
  0&0&0&0&0&0&0&1&0\\
  0&0&0&0&0&0&0&0&1},
g = \SmallMatrix{-i&0&0&1&0&0&0&0&0\\
  \tfrac{1}{2}&0&0&-\tfrac{1}{2}i&0&0&0&0&0\\
  0&i&0&0&-1&0&0&0&0\\
  0&\tfrac{1}{2}&0&0&-\tfrac{1}{2}i&0&0&0&0\\
  0&0&i&0&0&-1&0&0&0\\
  0&0&\tfrac{1}{2}&0&0&-\tfrac{1}{2}i&0&0&0\\
  0&0&0&0&0&0&1&0&0\\
  0&0&0&0&0&0&0&1&0\\
  0&0&0&0&0&0&0&0&1}.
$$
Applying $g$ to the space $\h$ we find a space spanned by the elements
\begin{align*}
  & -4e_{131}+2e_{162}-e_{241}+2e_{252}-4e_{353}+e_{463},\\
  & e_{141}+2e_{152}-e_{231}-\tfrac{1}{2} e_{262} -e_{363}-e_{453},\\
  & e_{123}+e_{342}-e_{561}
\end{align*}
which therefore span a second Cartan subspace of $\g_1$, not conjugate to
$\h$ under the group $\SL(6,\R)\times \SL(3,\R)$. 
\end{example}

\subsection{Classification of real semisimple orbits}

Now we consider the classification of the semisimple $\wG(\R)$-orbits in
$\g_1$ which we identify with the module $V$. We fix a Cartan subspace
$\h$ of $\g_1$. By Corollary \ref{cor:complc} $\h^\cC$ is a Cartan subspace
of $\g_1^\cC$. We let $W=W(\g^\cC,\theta)$ be the corresponding Weyl group.
As $\h^\cC$ is fixed under complex conjugation, also the group $W$ has a
conjugation $w \mapsto \upgam w$, induced by the conjugation on $G_0$. 
As in Section \ref{sec:clasC} we choose $p_1,\ldots,p_N\in \h^\cC$
such that for all $p\in \h^\cC$ we have that $W_p$ is conjugate in $W$ to
exactly one $W_{p_i}$. Then each semisimple $\wG$-orbit in $\g_1^\cC$ has a point
in exactly one of the $\h_{p_i}^{\cC,\circ}$. 

Throughout this section we use the following assumption:
\begin{equation}\label{eq:hyp}
\text{The $p_i$ are chosen such that $p_i\in \h$ (i.e., we have $\upgam p_i
=p_i$).} 
\end{equation}

\begin{remark}
If this hypothesis cannot be satisfied for an $\h_{p_i}^{\cC,\circ}$, that is, if
$\h_{p_i}^{\cC,\circ}$ does not contain real points, then we can do the following.
By the method outlined in the previous subsection (Theorem \ref{thm:cartan})
we can find real Cartan subspaces $\h_1=\h,\h_2,\ldots,\h_r$ such that each
Cartan subspace of $\g_1$ is $\wG(\R)$-conjugate to exactly one of the
$\h_i$. Furthermore we can find $g_i\in \wG(\R)$ with $\h_i = g_i(\h)$.
Then we check whether there is an $i$ such that $g_i(\h_{p_i}^{\cC,\circ})$ has
real points. If this holds for a certan $i$ then we work with the
set $g_i(\h_{p_i}^{\cC,\circ})$ instead of $\h_{p_i}^{\cC,\circ}$. If none of the
sets $g_i(\h_{p_i}^{\cC,\circ})$ has a real point then we conclude that the orbits
of the elements in $\h_{p_i}^{\cC,\circ}$ have no real point.
\end{remark}  

In order to simplify the notation a little bit we let $p$ be one of the $p_i$,
and set $\Ff = \h_p^{\cC,\circ}$. We now consider two problems:
\begin{enumerate}
\item Determine the elements of $\Ff$ whose $\wG$-orbits have real points, and
  for each of those orbits determine a real point.
\item For an orbit that has real points find representatives of the
  $\wG(\R)$-orbits contained in it.  
\end{enumerate}  

As in Section \ref{cor:complc} we consider the group $\Gamma_{p} =
\Nn_W(W_{p})/W_{p}$. Because of our assuption \eqref{eq:hyp} we have $\upgam p = p$,
so that $W_p$ and $\Nn_W(W_{p})$ are both stable under the conjugation of $W$.
Hence $\Gamma_p$ inherits the conjugation from $W$. 
We also set
\begin{align*}
\Zm_{\wG}(\Fm) & = \{ g\in \wG \mid gq=q \text{ for all } q\in \Fm\},\\
\Nm_{\wG}(\Fm) & = \{ g\in \wG \mid gq\in \Fm \text{ for all } q\in \Fm\}.
\end{align*}

Define a map $\varphi : \Nm_{\wG}(\Fm)\to \Gamma_p$ in the following way.
Let $g\in \Nm_{\wG}(\Fm)$, then $g\cdot p = q\in \Fm$. So by Corollary
\ref{cor:Gampconj} there is a $w\in \Nn_W(W_p)$ such that $wp=q$. Set
$\varphi(g) = wW_p \in \Gamma_p$. Note that this is well-defined: let $w'\in
\Nn_W(W_p)$ also satisfy $w'p=q$; then $(w')^{-1}w \in W_p$ so that
$w'W_p = wW_p$.

The following lemma and proposition have been proved in \cite{bgl} for the
special case considered in that paper. Here we prove these results under
our more general assumptions. However, the proofs are very similar to those
in \cite{bgl}. 

\begin{lemma}\label{lem:phihom}
$\varphi$ is a surjective group homomorphism with kernel $\Zm_{\wG}(\Fm)$.
Moreover, for $g\in \Nm_{\wG}(\Fm)$ and $q\in \Fm$ we have $g\cdot q =
\varphi(g)\cdot q$.   
\end{lemma}

\begin{proof}
First we claim the following: let $g\in \Nm_{\wG}(\Fm)$ and $w\in \Nm_W(W_p)$
be such that $g\cdot p = w\cdot p$, then $g\cdot q = w\cdot q$ for all $q\in
\Fm$. Indeed, there is a $\hat w\in \Nm_{\wG}(\h^\cC)$ that induces $w$.
Then $\hat w^{-1} g\in \Zz_{\wG}(p)$. By Corollary \ref{cor:cen} this yields
$\hat w^{-1} g\in \Zz_{\wG}(q)$, so that $g\cdot q = \hat w \cdot q = w\cdot q$.

Now let $g_1,g_2\in \Nm_{\wG}(\Fm)$ and let $w_1,w_2\in \Nm_W(W_p)$ be such
that $w_i\cdot p = g_i \cdot p$ for $i=1,2$. Then by our claim it follows that
$g_1g_2\cdot p = g_1\cdot (g_2\cdot p) = g_1\cdot (w_2\cdot p) = w_1\cdot
(w_2\cdot p) = w_1w_2\cdot p$. Hence $\varphi(g_1g_2) = \varphi(g_1)
\varphi(g_2)$. We see that $\varphi$ is a group homomorphism.

A $g\in  \Nm_{\wG}(\Fm)$ lies in the kernel of $\varphi$ if and only if
$g\cdot p =p$. By our claim this implies that $g\cot q = q$ for all $q\in \Fm$.
We conclude that $\Zm_{\wG}(\Fm)$ is the kernel of $\varphi$.

Let $w\in \Nm_W(W_p)$. Then there is a $\hat w \in \Nm_{\wG}(\h^\cC)$ such that
$\hat w\cdot q = w\cdot q$ for all $q\in \h^\cC$. In particular $\hat w \in
\Nm_{\wG} (\Fm)$. From the definition if $\varphi$ it is clear that
$\varphi(\hat w) = w$. So $\varphi$ is surjective.

The last statement of the lemma follows immediately from our claim above. 
\end{proof}  

The next proposition says how we can attack the first of our problems, in
the important special case where $\Ho^1 \wG = 1$. 

\begin{proposition}
Suppose that $\Ho^1 \wG=1$. Write $\Ho^1 \Gamma_p = \{ [\delta_1],\ldots,
[\delta_s]\}$, where $\delta_i\in \Gamma_p$ are cocycles. Suppose further
that for $1\leq i\leq s$ there is a cocycle $n_i \in \Nm_{\wG}(\Fm)$ with
$\varphi(n_i) = \delta_i$. Let $g_i\in \wG$ be such that $g_i^{-1}
\upgam g_i = n_i$. Let $q\in \Fm$ and let $\Oo = \wG\cdot q$ be its orbit.
Then $\Oo$ has real points if and only if there is a $q'\in \Oo\cap \Fm$
with $\upgam q' =  \delta_i^{-1} q'$ for an $i$ with $1\leq i\leq s$.
In that case $g_i\cdot q'$ is a real point of $\Oo$.
\end{proposition}

\begin{proof}
First note that the elements $g_i$ exist because $\Ho^1 \wG = 1$.

Suppose that $\Oo$ has a real point $g\cdot q$ for a $g\in \wG$. Then
$\upgam (g\cdot q ) = g\cdot q$. But that implies that $\upgam q =
(\upgam g)^{-1} g \cdot q = n^{-1} \cdot q$ with $n=g^{-1} \upgam g$.
We have $n \upgam n = g^{-1} \upgam g \upgam g^{-1} g = 1$ so that $n$ is a
cocycle.

Note that since $p$ is real, the set $\Fm$ is closed under $r\mapsto \upgam r$.
The above argument also shows that $\upgam q$ and $q$ are $\wG$-conjugate.
Since they both lie in $\Fm$ they are $\Gamma_p$-conjugate (Corollary
\ref{cor:Gampconj}) hence there is a $w\in \Nm_W(W_p)$ such that $w\cdot q =
\upgam q$. Let $\hat w\in \Nm_{\wG}(\Fm)$ be such that $\varphi(\hat w) = w$.
Then $n\hat w \in \Zz_{\wG}(q)$. By Corollary \ref{cor:cen} it follows that
$n\hat w \in \Zz_{\wG}(r)$ for all $r\in \Fm$. In other words, $n^{-1}\cdot r =
\hat w \cdot r$ for all $r\in \Fm$. Therefore $n$ lies in $\Nm_{\wG}(\Fm)$.

Set $\delta = \varphi(n)$. As $\varphi$ is $\Gamma$-equivariant we have
that $\delta$ is a cocycle in $\Gamma_p$. So there is an $i$ and a $\beta\in
\Gamma_p$ such that $\delta = \beta^{-1} \delta_i \upgam \beta$. Then
$\upgam q = \delta^{-1}\cdot q = \upgam \beta^{-1} \delta_i^{-1} \beta q$.
So if we set $q' = \beta q$, then $q'\in \Oo\cap \Fm$ and
$\upgam q' = \delta_i^{-1} q$.

Conversely, if $\upgam q' = \delta_i^{-1} q'$ then also $\upgam q' = n_i^{-1} q'$.
The latter is equivalent to $g_i q' = \upgam (g_iq')$. In particular,
$\Oo$ has a real point which is $g_1q'$. 
\end{proof}

So in order to find the elements of $\Fm$ whose orbits have real points we
determine the sets $C_i = \{ q\in \Fm \mid \upgam q = \gamma_i^{-1} q\}$.
Then the union of the $C_i$ is exactly the set that we require. Moreover,
if we set $R_i = \{ g_iq \mid q\in C_i\}$. Then the elements in $R_i$ are real
points of the orbits of the elements in $C_i$. (Note, however, that $R_i$ is
not necessarily contained in $\h^\cC$.)

Now we consider the second problem.
let $\Oo$ be a semisimple $\wG$-orbit, with given real point $q$. 
To determine the $\wG(\R)$-orbits contained in $\Oo$ we use Theorem
\ref{thm:cohom}. Then the $\wG(\R)$-orbits in $\Oo$ are in bijection with
$\ker[\Ho^1 \Zz_{\wG}(q)\to \Ho^1 \wG]$.
Here we note one important fact: the sets $R_i$ are, in general, infinite,
but all elements of $R_i$ have the same centralizer in $\wG$ by Corollary
\ref{cor:cen}. Therefore we have to compute $\ker[\Ho^1 \Zz_{\wG}(q)\to \Ho^1
\wG]$ only for one element $q$ of $R_i$. 

\section*{Acknowledgement} The authors  wish to thank Mikhail Borovoi  for   fruitful  discussions over  their  joint  works \cite{bgl}, \cite{bgl21} which  stimulate  the  results  in the  present  article.

\end{document}